\documentclass[10pt,a4paper]{amsart}
\usepackage{amscd,amssymb,amsopn,amsmath,amsthm,graphics,amsfonts,enumerate,verbatim,calc}
\usepackage[dvips]{graphicx}

\usepackage{mathpazo}
\usepackage{color}
\usepackage{latexsym}
\usepackage{amsthm,amsfonts,amssymb,mathrsfs}
\usepackage{rotating}
\usepackage[leqno]{amsmath}
\usepackage{xspace}
\usepackage[all]{xy}
\usepackage{longtable}
\headsep=1cm \numberwithin{equation}{section}
\newtheorem{theorem}{Theorem}[section]
\newtheorem{lemma}[theorem]{Lemma}

\newtheorem{proposition}[theorem]{Proposition}
\newtheorem{corollary}[theorem]{Corollary}

\theoremstyle{definition}

\theoremstyle{remark}
\newtheorem{remark}[theorem]{Remark}
\newtheorem{fact}[theorem]{Fact}
\newtheorem{example}[theorem]{Example}

\newtheorem{question}[theorem]{Question}

\newcommand{\HH}{\operatorname{H}}

\newcommand{\Ass}{\operatorname{Ass}}
\newcommand{\assh}{\operatorname{Assh}}

\newcommand{\Spec}{\operatorname{Spec}}

\newcommand{\rad}{\operatorname{rad}}

\newcommand{\Ht}{\operatorname{ht}}
\newcommand{\pd}{\operatorname{p.dim}}

\newcommand{\Syz}{\operatorname{Syz}}

\newcommand{\rank}{\operatorname{rank}}

\newcommand{\V}{\operatorname{V}}

\newcommand{\Ext}{\operatorname{Ext}}
\newcommand{\Se}{\operatorname{S}}

\newcommand{\R}{\operatorname{R}}
\newcommand{\Supp}{\operatorname{Supp}}

\newcommand{\Tor}{\operatorname{Tor}}
\newcommand{\COH}{\operatorname{H}}

\newcommand{\Ann}{\operatorname{Ann}}

\newcommand{\depth}{\operatorname{depth}}

\newcommand{\coker}{\operatorname{coker}}

\newcommand{\lo}{\longrightarrow}
\newcommand{\fm}{\frak{m}}
\newcommand{\fp}{\frak{p}}
\newcommand{\fq}{\frak{q}}

\newcommand{\fn}{\frak{n}}

\begin{document}

\author[]{mohsen asgharzadeh}

\address{}
\email{mohsenasgharzadeh@gmail.com}

\title[ ]
{on the  dimension of syzygies}

\subjclass[2010]{ Primary 13D02  Secondary 13C15; 13D45}
\keywords{ Betti numbers; Buchsbaum rings; dimension theory; length function; reduced rings; support; syzygy modules}
\dedicatory{}
\begin{abstract}
In this  note we compute length, support and  dimension of syzygy modules of certain modules.
This partially answers questions asked by Huneke et al.
\end{abstract}

\maketitle

\section{Introduction}
In this note $(R,\fm,k)$ is a commutative noetherian local ring of dimension $d>0$ and $0\neq M$ is a finitely generated $R$-
module. The notation $\pd(-)$ stands for the projective dimension  and $\lambda(-)$ is the length function.  Let $i\in \mathbb{N}_0$. The $i^{th}$ \textit{betti number} of $M$ is given by $\beta_i(M):=\dim_k(\Tor^R_i(k,M))$. If there is no danger of confusion we will use $\beta_i$ instead of $\beta_i(M)$. A minimal free resolution of $M$ is of  the form
$\cdots \lo R^{\beta_{i+1}}\stackrel{f_{i+1}}\lo  R^{\beta_{i}}\lo\cdots\lo  R^{\beta_{0}}\lo M\lo 0.$
The $i^{th}$ \textit{syzygy} module of $M$ is $\Syz_i(M) := \coker(f_{i+1})=\ker(f_{i-1})$ for all $i>0$.
Computing numerical invariants of the syzygy modules is of some interest for a variety of reasons. Our first aim is  to compute the length of syzygy modules:

\begin{question}(See \cite[Question 1.2]{h})
 Let $M$ be such that $\pd_R(M) = \infty$ and $\lambda(M) <\infty$. Is $\lambda(\Syz_i(M))=\infty$ for all $i > d + 1$?
\end{question}

The assumption $d>0$ is really needed:
Indeed, let $R$ be a zero-dimensional local ring which is not a field. Then $\pd(R/\fm)=\infty$ and
each of its syzygy modules are nonzero. Since $R$ is zero-dimensional, any finitely generated module
is of finite length. So, $d$ should be positive. There are few progress concerning Question 1.1. Let us recall
an achievement from literature.
Recently, Huneke and his coauthors showed that Question 1.1 is true over 1-dimensional \textit{Buchsbaum} rings. Also, they showed the requirement of $i > d+1$  is necessary (over 1-dimensional rings):

\begin{example}
 Let $R := k[[x, y]]/(x^2, xy)$ and $M := R/(y)$. Then
$\pd_R(M) =\infty$, $\lambda(M)<\infty$, and $\lambda(\Syz_2(M))<\infty$. We should remark that the module  $N = R/(x)$ does not do the same job.
\end{example}

In support of Question 1.1 we present four observations. The first one  drops the dimension restriction from the Buchsbaum rings:

\textbf{Observation A.}
Let  $(R,\fm)$ be a Buchsbaum ring of dimension $d>1$,  $\pd_R(M) = \infty$ and $\lambda(M) <\infty$. Then $\Supp(\Syz_i(M))=\Spec(R)$ for all   $i >0$. In particular, $\lambda(\Syz_i(M))=\infty$.

We show a little more, please see Corollary \ref{m1}.
  We observed in Example 1.2 that  second syzygy module  (of a finite length module) may be of finite length.  If we focus on the simple module the story will changes:\\

\textbf{Observation B.} We reprove a result of Okiyama by a short argument:
\begin{enumerate}
\item[i)] If $R$ is  regular, then $\Syz_i(R/\fm)=0$  for all $i> d$ and $\dim(\Syz_i(R/\fm))=d$ for all $i\leq d$.
\item[ii)] If $R$ is not regular, then $\Supp(\Syz_i(R/\fm))=\Spec(R)$ for all $i>0$.
\end{enumerate}
Against Okiyama, we avoid Tate's approach of homology of local rings.
Our second aim is  to investigate the following question:

\begin{question}\label{1.5}(See \cite{ext} and \cite{b})
Is $\dim(\Syz_i(M))$ constant for all $i\gg 0$?
\end{question}

 To find a connection between Question 1.1 and Question \ref{1.5} let us revisit  Example 1.2, where  we observed that $\lambda(\Syz_1(M))=\infty$.
In fact, $ \Syz_1(M) =yR$. Also, $\Ann(yR)=xR$. Thus, $\Supp(\Syz_1(M))=\V(xR)=\Spec(R).$  In fact, the following  extends and corrects some  results from literature:\\

\textbf{Observation C.}\label{o}
Let $d>0$ and $0\neq M$ be a finite length module of infinite projective dimension. Then, for all $r\geq 0$ the following conditions are equivalent:
\begin{enumerate}
\item[i)]  $\Supp(\Syz_{r+1}(M))=\Spec(R)$,
\item[ii)] $\dim(\Syz_{r+1}(M))=\dim R$,
\item[iii)] $\lambda(\Syz_{r+1}(M))=\infty$.
\end{enumerate}

By $\assh(R)$ we mean the set of all prime ideals  $\fp$ such that $\dim (R)=\dim (R/\fp)$.
Here is an immediate application:
Let $R$ be a ring such that $\Ass(R)=\assh(R)$ (e.g. $R$ is Cohen-Macaulay), $\lambda(M)<\infty$ and $\pd_R(M) = \infty$. Then $\Supp(\Syz_i(M))=\Spec( R)$ for all $i>0$.
In particular, Question 1.1 and Question \ref{1.5} are true over integral domains. To see more applications of Observation C, please see Corollary \ref{1,3}
and Corollary \ref{mis}. Of course the last four results only work  for finite length modules, please see Example \ref{1.6}.
Section 2 is devoted to the proof of Observation C. In \S 3  we show:\\

\textbf{Observation D.}
Let $R$ be a  reduced local ring  and $M$ a finite length module of infinite projective dimension. Then $\Supp(\Syz_{i}(M))=\Spec (R)$ for all $i> 0$.

\S 4 is devoted to the proof of Observation A and Observation B. \S 5 runs Question 1.1 over some rings.
 Complete reduced rings are quotients of  regular local rings by a radical ideal.  Let us define the following related class of rings:
A ring is  called \textit{weakly reduced}
if it is quotient of a local ring  by a nonzero and  integrally closed ideal.
Recall that $I\subseteq \overline{I}\subseteq \rad(I)$. Thus, any complete reduced ring is weakly reduced. Also, we show:

\begin{corollary}\label{1.66}
Let $R$ be a weakly reduced ring of dimension $d>1$ and $M$ a finite length module of infinite projective dimension. Then $\Supp(\Syz_{i}(M))=\Spec (R)$ for all $i> 0$.
\end{corollary}

\section{From Question 1.1 to  Question \ref{1.5}}

 A finitely generated module $M$ is called \textit{locally free on the punctured spectrum}
if $M_{\fq}$ is free over $ R_{\fq}$ for all $\fq\in \Spec(R)\setminus\{ \fm\}$.
For example, any finite length module is locally free.
As another example:

\begin{remark}
Let $R$ be a 1-dimensional reduced local ring. Then  any finitely generated module  $M$ is locally free.
Indeed, since $R$ is reduced, $R$ satisfies  Serre's condition  $(\R_0)$. This means that  $R_{\fp}$ is  regular for all $\fp\in\Spec(R)\setminus\{\fm\}$. Zero-dimensional regular local rings
are field. In particular, any module over a such ring is a vector space.
\end{remark}

\begin{lemma}\label{d}
Let $(R,\fm)$ be equi-dimensional  and $M$  be  locally free on the punctured spectrum. Then either $\dim(\Syz_i(M)) = \dim R$ or $\lambda(\Syz_i(M))<\infty$.
\end{lemma}

In the next argument, we are not in a position to drop the equi-dimensional assumption.

\begin{proof}
If $d=\dim R\leq 1$ there is nothing to prove. In particular, we may assume that $d>0$ and that $\Syz_i(M)\neq 0$. If $\min\left(\Supp(\Syz_i(M))\right)\subset \min(R)$ were be the case, then we should have $\dim(\Syz_i(M)) = \dim R$, because $R$ is euqi-dimensional. Thus, we can assume $\min(\Supp(\Syz_i(M)))\nsubseteqq \min(R)$.
Let $\fp\in\min(\Supp(\Syz_i(M)))\setminus \min(R)$. We claim that $\fp$ is the maximal ideal. Suppose on the contrary that
$\fp\neq \fm$. To search a contradiction, we look at $\textbf{F}\to M\to 0$ the minimal free resolution of $M$. Since $M$ is locally free, $M_{\fp}$ is free. Consequently, $\textbf{F}_{\fp}\to M_{\fp}\to 0$ splits. It turns out that $\Syz_i(M)_{\fp}$ is free. Therefore, $\dim (\Syz_i(M)_{\fp})=\dim (R_{\fp})>0$. Since $\fp\in \min(\Supp(\Syz_i(M)))$ we get to a contradiction.
\end{proof}

However, if $M$ is of finite length we are able to  drop the equi-dimensional assumption:

\begin{lemma}\label{o}
Let $d>0$ and $0\neq M$ be finite length module of infinite projective dimension. Then, for all $r\geq 0$ the following conditions are equivalent:
\begin{enumerate}
\item[1)]   $\lambda(\Syz_{r+1}(M))=\infty$,
\item[2)] $\sum_{i=0}^r(-1)^{r-i}\beta_{i}(M)>0$,
\item[3)] $\Supp(\Syz_{r+1}(M))=\Spec(R)$,
\item[4)]  $\dim(\Syz_{r+1}(M))=\dim R$.
\end{enumerate}
\end{lemma}

\begin{proof}
First we recall  a routine fact.
 Let $\fp\neq\fm$ be a prime ideal. Such a thing exists, because $d>0$. Keep in mind that $M$ is of finite length. Since $M_{\fp}=0$, we have the following split exact sequence:$$0 \lo \Syz_{r+1}(M)_{\fp}\lo R^{\beta_{r}}_{\fp}\lo  \cdots\lo  R^{\beta_{0}}_{\fp}\lo 0.$$Since the sequence splits, $\Syz_{r+1}(M)_{\fp}$ is free.  So, $$\sum_{i=0}^{r}(-1)^{r-i}\beta_{i}(M)=\rank(\Syz_{r+1}(M)_{\fp})\geq 0\quad(\ast)$$
\begin{enumerate}
\item[$1)\Rightarrow 2)$]
Let $\fp\neq\fm$ be a prime ideal in $\Supp(\dim\Syz_{r+1}(M))$. By the assumption such a $\fp$ exists. Thus  $\Syz_{r+1}(M)_{\fp}$ is a nonzero free module. Therefore, $\rank(\Syz_{r+1}(M)_{\fp})> 0$.
From $(\ast)$ we get that $\sum_{i=0}^r(-1)^{r-i}\beta_{i}(M)>0$, as claimed.
\item[$2)\Rightarrow 3)$]   Let $\fp\in \Spec(R)\setminus\{\fm\}$. By the assumption,
$\sum_{i=0}^{r}(-1)^{r-i}\beta_{i}(M)> 0$ . In view of $(\ast)$  we have $\rank(\Syz_{r+1}(M)_{\fp})>0$. Therefore, $\fp\in \Supp(\Syz_{r+1}(M))$. Thus, $\Spec(R)\setminus\{\fm\}\subset\Supp(\Syz_{r+1}(M))$. One has $\Syz_{r+1}(M)\neq 0$. Hence
 $\fm\in \Supp(\Syz_{r+1}(M))$. Consequently, $\Spec(R)\subset\Supp(\Syz_{r+1}(M)).$ The reverse inclusion always hold. So, $\Spec(R)=\Supp(\Syz_{r+1}(M))$ as claimed.
\item[$3)\Rightarrow 4)$] This is clear.
\item[$4)\Rightarrow 1)$]  Since $d>0$, a finitely generated module of dimension $d$ is of infinite length. Thus, $\lambda(\Syz_{r+1}(M))=\infty$.
\end{enumerate}
\end{proof}

\begin{corollary}\label{1,3}
Let $M$ be a finite length module and of infinite projective dimension over a  1-dimensional local ring $R$. Then $\Supp(\Syz_{1}(M))=\Supp(\Syz_{3}(M))=\Spec(R)$.
 \end{corollary}

\begin{proof}
 In view of \cite[Corollary 5.10]{h}, we  see   $\lambda(\Syz_{1}(M))=\lambda(\Syz_{3}(M))=\infty$.
We apply this along with Lemma \ref{o} to conclude that $\Supp(\Syz_{1}(M))=\Supp(\Syz_{3}(M))=\Spec(R).$
\end{proof}

For the simplicity of the reader we cite:
\begin{lemma}(See \cite[Proposition 5.5]{h}) \label{desend}
Let $R$ be a local ring of positive dimension. Suppose there is
an $R$-module $M$ of infinite projective dimension and finite length such that
$\lambda(\Syz_{i+1}(M))<\infty$  for some fixed $i > 0$. If $\beta_i(M) \geq \beta_{i-1}(M)$, then
$\lambda(\Syz_{i-1}(M))<\infty$.
\end{lemma}

\begin{corollary}\label{mis}
Let $0\neq M$ be a finite length module such that $\beta_i(M)\leq\beta_{i+1}(M)$ for all $i>0$. Then $\Supp(\Syz_{2i+1}(M))=\Spec(R)$ for all $i\geq0$.
 \end{corollary}

\begin{proof}
We may assume that $d>0$.
First note that $\pd(M)=\infty$, because $\beta_i(M)\leq\beta_{i+1}(M)$ for all $i>0$. Clearly, $\lambda(\Syz_1(M))=\infty$. This follows by looking at the following short exact sequence $0\to\Syz_1(M)\to R^{\beta_0}\to M\to 0$. Due to  Lemma \ref{o} $\Supp(\Syz_{1}(M))=\Spec(R)$.
Suppose  on the contrary that  $\Supp(\Syz_{2i+1}(M))\neq\Spec(R)$ for some $i>0$. By revisiting Lemma \ref{o}, we see that $\lambda(\Syz_{2i+1}(M))<\infty$. We apply this along with  Lemma \ref{desend} to observe that
$\lambda(\Syz_{2i-1}(M))<\infty$. If $2i-1\neq 1$ we can repeat the argument to observe that $\lambda(\Syz_{1}(M))<\infty$, a contradiction.
\end{proof}

Let $d(M)$ be the smallest integer $\ell$ such that $\dim(\Syz_i(M))$ is constant for all $i>\ell$.
Let $\mathcal{C}$ be a class of finitely generated modules. Suppose $d(M)$ is finite for all $M\in \mathcal{C}$.
Is $\sup\{d(M):M\in \mathcal{C}\}<\infty$?
The classes that we are interested on it are the class of finitely generated modules, the class of finite length modules and the class
of modules with fixed  numerical invariants.

\section{Dealing with reduced rings}

In the  Cohen-Macaulay case  and for all $i>\dim R$  the following fact
is in \cite{b}.

\begin{fact}(Okiyama)\label{oki}
Let $R$ be a ring such that $\Ass(R)=\assh(R)$ (e.g. $R$ is Cohen-Macaulay or $R$ is a domain) and $\pd_R(M) = \infty$. Then $\dim(\Syz_i(M))=\dim R$ for all $i>0$.
 \end{fact}

\begin{proof}
By looking at the following exact sequence $0\to \Syz_i(M)\to R^{\beta_{i-1}}\to \Syz_{i-1}(M) \to 0$ we observe  that $\Ass(\Syz_i(M))\subset\Ass(R)$. Since $\Syz_i(M)\neq 0$,  $\Ass(\Syz_i(M))\neq\emptyset$.  Let $\fp\in \Ass(\Syz_i(M))$. Then  $\fp\in \Ass(R)=\assh(R)$. By definition, $\dim R/\fp=\dim R$. So, $\dim(\Syz_i(M))=\dim R$ as claimed.
\end{proof}

\begin{corollary}\label{cm}
Let $R$ be a ring such that $\Ass(R)=\assh(R)$, $\lambda(M)<\infty$ and $\pd_R(M) = \infty$. Then $\Supp(\Syz_i(M))=\Spec( R)$ for all $i>0$.
 \end{corollary}

\begin{proof} By Fact \ref{oki} $\dim(\Syz_i(M))=\dim R$ for all $i>0$.
It is enough to apply Lemma \ref{o}.
\end{proof}

 The finite length assumption in Corollary \ref{mis}, Corollary \ref{1,3}, Lemma \ref{o}, and  Corollary \ref{cm} is important:

 \begin{example}\label{1.6}
  We look at the Cohen-Macaulay ring
$R:=k[[X,Y]]/(XY)$ and the infinite length module $M:=R/xR$. The following holds: \begin{enumerate}
\item[i)]  One has $\Ass(R)=\assh(R)$,
\item[ii)] $\Supp(\Syz_i(M))\neq \Spec(R)$ for all $i>0$,
\item[iii)] $\dim(\Syz_{i}(M))=\dim( R)$ for all $i>0$.
\end{enumerate}
\end{example}

\begin{proof}
Clearly, $\Ass(R)=\{(x),(y)\}=\min(R)=\assh(R)$. Also, $\Supp(M)=\{(x),(x,y)\}$ and that  $\dim M =1$. This implies that $\lambda(M)=\infty$.
The minimal free resolution of $M$ is given by  $\ldots\stackrel{x}\lo R\stackrel{y}\lo R\stackrel{x}\lo R\to M \to 0.$ Then \begin{equation*}
\Syz_{i}(M)= \left\{
\begin{array}{rl}
R/xR & \  \   \   \   \   \ \  \   \   \   \   \ \text{if } i\in2\mathbb{N}_0\\
R/yR & \  \   \   \   \   \ \  \   \   \   \   \ \text{if } i\in2\mathbb{N}_0+1
\end{array} \right.
\end{equation*} So, \begin{equation*}
\Supp(\Syz_{i}(M))= \left\{
\begin{array}{rl}
\V(xR) & \  \   \   \   \   \ \  \   \   \   \   \ \text{if } i\in2\mathbb{N}_0\\
\V(yR) & \  \   \   \   \   \ \  \   \   \   \   \ \text{if } i\in2\mathbb{N}_0+1
\end{array} \right.
\end{equation*} This shows that $\Supp(\Syz_i(M))\neq \Spec(R)$ for all $i>0$ and that  $\dim(\Syz_{i}(M))=\dim( R)$ for all $i>0$.
\end{proof}

\begin{lemma}\label{depth}
Let $R$ be of positive depth. Then Question 1.1 is   true. In fact,  if $M$ is such that $\pd_R(M) = \infty$ and $\lambda(M) <\infty$, then $\Supp(\Syz_{i}(M))=\Spec( R)$ for all $i > 0$.
\end{lemma}

\begin{proof}
Recall that $\depth$ of a finitely generated module $L$ is defined by $\inf\{j\geq 0:\Ext^j_R(R/\fm,L)\neq 0\}$.
Let $i>0$ and look at the exact sequence  $0\to \Syz_i(M)\to R^{\beta_{i-1}}\to \Syz_{i-1}(M)\to 0. $ Note that $\depth (R^{\beta_{i-1}})>0$.
Apply the long exact sequence of Ext-modules $\Ext^{\ast}_R(R/\fm,-)$
  to deduce  that $$\depth (\Syz_i(M))\geq\inf \{\depth (R^{\beta_{i-1}}),\depth(\Syz_{i-1}(M))+1\}>0.$$
Since $\depth$  of any nonzero finite length module is zero, we get that $\lambda(\Syz_i(M))=\infty$. We conclude from Lemma \ref{o}  that
$\Supp(\Syz_{i}(M))=\Spec( R)$ for all $i > 0$.
\end{proof}

\begin{corollary}\label{Cdepth}
Let $R$ be equi-dimensional and of positive depth. If $M$ is locally free over punctured spectrum, then $\dim(\Syz_i(M))$ is constant for all $i\gg 0$.
\end{corollary}

\begin{proof}Suppose first that $\pd(M)$ is finite.
Then $ \Syz_i(M)=0$  for all $i\gg 0$. Then, without loss of the generality may assume that $\pd(M)=\infty$.
Recall that $$\depth (\Syz_i(M))\geq\inf \{\depth (R^{\beta_{i-1}}),\depth(\Syz_{i-1}(M))+1\}>0,$$i.e., $\lambda(\Syz_i(M))=\infty$.
We apply the equi-dimensional and the locally free assumption along  with  Lemma \ref{d}
to see that $\dim(\Syz_i(M)) = \dim R$ for all $i>0$.
\end{proof}

 \begin{corollary}\label{red1}
Let $R$ be a  reduced local ring  and $M$ a finite length module of infinite projective dimension. Then $\Supp(\Syz_{i}(M))=\Spec (R)$ for all $i> 0$.
\end{corollary}

\begin{proof}
We may assume that $\dim R>0$.
Reduced rings satisfy in the Serre's condition $\Se_1$. One may read this as follows: $\depth(R_{\fp})\geq\min\{1,\Ht(\fp)\}$ for all $\fp\in \Spec(R)$. We apply this for
the maximal ideal to observe that $\depth(R)>0$. Now Lemma \ref{depth} yields the claim.
\end{proof}

Revisiting Example \ref{1.6}, we observe that the finite length assumption in Corollary \ref{red1} is  needed.

\section{Looking through  Buchsbaum  glasses}

By $\HH^0_{\fm}(R)$, we mean the elements of $R$ that  are annihilated by some  power  of $\fm$.

\begin{lemma}\label{v}(Vanishing result)
 Let $M$ be  locally free over punctured spectrum that
$\lambda(\Syz_{i+1}(M))<\infty$
for some fixed $i > 0$. Then
$\Tor^R_i(M,R/\COH^0_{\fm}(R))= 0.$
\end{lemma}

\begin{proof}
The proof in the case $M$ is of finite length is in  \cite[Lemma 5.2]{h}. Again,
such a proof works for locally free modules.
\end{proof}

\begin{lemma}\label{new}
Let $M$ be of finite length. Then $\lambda(\Syz_i(M))=\infty$  for all $1\leq i\leq d$ provided $\Syz_i(M)\neq0$. In fact $\Supp(\Syz_{i}(M))=\Spec(R)$.
\end{lemma}

\begin{proof}
Suppose on the contradiction that $\lambda(\Syz_i(M))<\infty$ for some $1\leq i\leq d$. Among these, we look at the minimal one, and denote it again by $i$.
We use the \textit{new intersection theorem} \cite{int} along with the following complex of free modules
with finite length homologies $0 \to R^{\beta_{i-1}}\to \cdots\to  R^{\beta_{0}}\to  0$ to deduce that
 $i-1\geq \dim R$. This excluded by the assumption. Now, the proof of Lemma \ref{o} shows that
$\Supp(\Syz_{i}(M))=\Spec(R)$.
\end{proof}

We need to recall the following result:  Let $R$ be a noetherian ring  and $0\neq I$
 an ideal with a
finite free resolution. Then $I$ contains an $R$-regular element, see \cite[Corollary 1.4.7]{BH}.

\begin{corollary}\label{vc} Let $(R,\fm,k)$ be a $d$-dimensional ring  with $d>0$.
Then $\Syz_i(R/\fm)$ is of infinite length provided $\Syz_i(R/\fm)$ is nonzero  for some fixed $i > 0$. In fact,  $\Supp(\Syz_i(R/ \fm))=\Spec(R)$.
\end{corollary}

\begin{proof}In view of $0\to \Syz_1(R/\fm)\to R^{\beta_0}\to R/\fm\to 0$ we observe that $\Syz_1(R/\fm)$ is of infinite length. Then we may assume that $i>1$.  Suppose on the contradiction that $\Syz_{i+1}(R/\fm)$ is of finite length for some $i>0$. Then by the vanishing result
we have $\Tor^R_i(k,R/\COH^0_{\fm}(R))= 0$. Keep in mind that $\pd(R/\COH^0_{\fm}(R))=\sup\{j\geq 0:\Tor^R_{j}(k,R/\COH^0_{\fm}(R))\neq 0\}.$ Consequently, $\COH^0_{\fm}(R)$ has a finite free resolution. In view of Lemma \ref{depth}, we may and do assume that $\COH^0_{\fm}(R)\neq 0$. Also, $\COH^0_{\fm}(R)\neq R$,
because $d>0$. We can apply \cite[Corollary 1.4.7]{BH} to conclude  that  $\COH^0_{\fm}(R)$ contains an $R$-regular element. Since each element of $\COH^0_{\fm}(R)$
is annihilated by some power of $\fm$ we get to a contradiction. By the proof of  Lemma \ref{o}, $\Supp(\Syz_i(R/ \fm))=\Spec(R)$.
\end{proof}

Let $(R,\fm)$ be a local ring.   Recall that a sequence $x_{1},\ldots ,x_{r} \subset \fm$ is called a \textit{weak sequence} if $\fm ((x_{1},\cdots ,x_{i-1})\colon x_{i})\subseteq(x_{1},\cdots ,x_{i-1})$ for all $i$.
The ring $R$ is called \textit{Buchsbaum} if every system of parameters is a weak sequence.
Now, let $R$ be  Buchsbaum.  Recall from \cite[Lemma 2.4]{bus} that
$\fm H^i_{\fm}(R)=0$ for all $i\neq\dim R$. The converse of this is not true, see \cite[Page 75]{bus}.

 \begin{proposition} \label{m}
Let $(R,\fm,k)$ be a $d$-dimensional ring for which $\fm \COH^0_{\fm}(R)=0$ (e.g. $R$ is Buchsbaum) and that $d>1$. Let $M$ be finite length
such that $\pd_R(M) = \infty$. Then $\Supp(\Syz_i(M))=\Spec(R)$   for all $i> 0$. In particular, $\lambda(\Syz_i(M)) = \infty$.
\end{proposition}

\begin{proof}
 In view of  Lemma \ref{depth} we can assume that
$\depth (R)=0$. In particular, $\COH^0_{\fm}(R)\neq 0$.
Thus, $\COH^0_{\fm}(R)$ is a nonzero $k$-vector space. Suppose first that $i>2$. Suppose on the contradiction that $\Syz_i(M)$ is of finite length.
We can apply Lemma \ref{v}: $\Tor^R_{i-1}(M,R/\COH^0_{\fm}(R))=0$, because $i-1>0$.
Since $i-2>0$ we have $\Tor^R_{i-2}(M,R)=\Tor^R_{i-1}(M,R)=0$.  The long exact sequence induced by $0\to \COH^0_{\fm}(R)\to R \to  R/\COH^0_{\fm}(R)\to 0$ implies that
$\Tor^R_{i-1}(M,R/\COH^0_{\fm}(R))\simeq\Tor^R_{i-2}(M,\COH^0_{\fm}(R))$. Let us display things:
$$0=\Tor^R_{i-1}(M,R/\COH^0_{\fm}(R))\simeq\Tor^R_{i-2}(M,\COH^0_{\fm}(R))\simeq \bigoplus_{\textit{non empty}} \Tor^R_{i-2}(M,k).$$ Recall that   $\pd(M)=\sup\{j\geq 0:\Tor^R_{j}(k,M)\neq 0\}.$ Consequently, $\pd(M)<\infty$. This excluded by the assumption.
This contradiction yields that $\Syz_i(M)$ is  of infinite length for all $i>2$.
Clearly, $\lambda(\Syz_1(M))=\infty$. This follows by looking at the following short exact sequence $0\to\Syz_1(M)\to R^{\beta_0}\to M\to 0.$
The only $j$ that has  chance to $\lambda(\Syz_j(M))<\infty$ is $j=2$. This excluded from Lemma \ref{new}. Here is a place that we use the assumption $d>1$. In sum,
$\lambda(\Syz_i(M))=\infty$ for all $i > 0$.
Finally, we deduce from Lemma \ref{o} that $\Supp(\Syz_i(M))=\Spec(R)$  for all $i> 0$.
\end{proof}

In a similar vein we have:

 \begin{corollary} \label{m1}
Let $(R,\fm,k)$ be a $d$-dimensional ring for which $\COH^0_{\fm}(R)$ is a  $k$-vector space (e.g. $R$ is Buchsbaum) and $d>0$. Let $M$ be locally free
such that $\pd_R(M) = \infty$. Then $\lambda(\Syz_i(M)) = \infty$ for all $i>2$. Suppose in addition  that $R$ is  equidimensional. Then $\dim(\Syz_i(M)) = \dim R$ for all $i>2$.
\end{corollary}

\begin{proof}
The claim in the case $d=1$ follows from \cite[Proposition 5.3]{h} under the assumption $\lambda(M)<\infty$. The same argument works for locally free modules.
Then we may assume that $d>1$. Now, the first desired claim is in Proposition \ref{m}.
If $R$  is equi-dimensional, we deduce from  Lemma \ref{d} that $\dim(\Syz_i(M)) = \dim R$ for all $i>2$.
\end{proof}

\section{Concluding examples}
We start by proving Corollary \ref{1.66}. First, we recall the main point for dealing with weakly reduced rings:

\begin{fact}\label{integral}(See \cite[Corollary 1.2]{ram2})
Let $S$ be a local ring with an integrally closed ideal $I$. Suppose $R:=S/I$ is of zero depth.
Let $M$  be a finitely generated $R$-module.
Then $\{\beta_i(M)\}$ is not decreasing.
\end{fact}

 Now, we extend Observation D in the following sense:

\begin{corollary}
Let $R$ be a weakly reduced local ring of dimension $d>1$ and $M$ a finite length module of infinite projective dimension. Then $\Supp(\Syz_{i}(M))=\Spec (R)$ for all $i> 0$.
\end{corollary}

\begin{proof}
 In view of Lemma \ref{depth} we may assume that $\depth (R)=0$. Thus, we are in a situation to  apply Fact \ref{integral}, i.e.,
$\{\beta_i(M)\}$ is not decreasing. We apply Corollary \ref{mis} to observe that $\Supp(\Syz_{2i+1}(M))=\Spec (R)$ for all $i> 0$.
Recall from Lemma \ref{new} that $\lambda(\Syz_{2}(M))=\infty$
(here we used the assumption $d>1$).
Suppose for some $i>1$  we have $\Supp(\Syz_{2i}(M))\neq\Spec(R)$.
Due to Lemma \ref{o}, $\lambda(\Syz_{2i}(M))<\infty$.
We apply this along with Lemma \ref{desend} to observe that
$\lambda(\Syz_{2i-2}(M))<\infty$. If $2i-2\neq 2$ we can repeat the argument to observe that $\lambda(\Syz_{2}(M))<\infty$ a contradiction.
Thus, $\Supp(\Syz_{2i}(M))=\Spec(R)$. Consequently, $\Supp(\Syz_{i}(M))=\Spec (R)$ for all $i> 0$.
\end{proof}

\begin{lemma}(See \cite[Proposition 2.1]{ram2}) \label{ram2}
Let $I$ be a  non-nilpotent ideal in a local ring $(S,\fn)$. Set $R:=\frac{S}{I\fn}$. Let $M$ be a finitely generated $R$-module. Then $M$
 has  increasing betti numbers.
\end{lemma}

\begin{example}
Look at the ring $R:=k[[X,Y]]/X(X,Y)^n$ for some $n>0$. Then $\lambda(\Syz_{i}(R/\fm^n))=\infty$ for all $i>0$.
In fact, $\Supp(\Syz_{i}(R/\fm^n))=\Spec(R)$.
\end{example}

\begin{proof}
 Clearly, $\lambda(\Syz_{1}(R/\fm^n))=\lambda(\fm^n)=\infty$. Let $i>0$ and apply Lemma  \ref{ram2} for $I:=X(X,Y)^{n-1}\lhd k[[X,Y]]$ to see $\beta_i(M) \geq \beta_{i-1}(M)$ for any $R$-module. Thus, Lemma \ref{desend}
 implies that $\lambda(\Syz_{2i+1}(R/\fm^n))=\infty$. Similarly, $\lambda(\Syz_{2i}(R/\fm^n))=\infty$ provided $\lambda(\Syz_{2}(R/\fm^n))=\infty$. Hence, things reduce to show $\lambda(\Syz_{2}(R/\fm^n))=\infty$.
One has $\COH^0_{\fm}(R)\supseteq xR $. This annihilated by $\fm^n$. Suppose on the contrary that $\lambda(\Syz_{2}(R/\fm^n))<\infty$.
In the light of  vanishing result (see Lemma \ref{v}) we deduce that
$\Tor^R_1(R/\fm^n,R/\COH^0_{\fm}(R))= 0.$ But  $$\Tor^R_1(R/\fm^n,R/\COH^0_{\fm}(R))= \frac{\fm^n\cap \COH^0_{\fm}(R)}{\fm^n \COH^0_{\fm}(R)}=\fm^n\cap \COH^0_{\fm}(R).$$To get a contradiction it is enough to note that $0\neq x^n\in\fm^n\cap(x)\subseteq \fm^n\cap\COH^0_{\fm}(R)$. Therefore, $\lambda(\Syz_{2}(R/\fm^n))=\lambda(\fm^n)=\infty$. We conclude from Lemma \ref{o} that $\Supp(\Syz_i(M))=\Spec (R)$ for all $i>0$.
\end{proof}

\begin{example}
Let $0\neq f$ be a non-unit power series in $k[[X_1,\ldots,X_n]]$  with $n>2$ and let $R:=k[[X_1,\ldots,X_n]]/f\fm$. Let $M$ be locally free and of infinite projective dimension. Then $\dim (\Syz_i(M))=\dim R$  for all $i>2$. If $\lambda(M)<\infty$, then $\dim (\Syz_i(M))=\dim R$  for all $i>0$.
\end{example}

In the above example we have $\Ass(R)\neq\assh(R)$, because $\fm\in \Ass(R)\setminus\assh(R)$. Also, $R$ is not reduced, e.g. $f^2=0$ and $f\neq 0$.

\begin{proof}
We have $\COH^0_{\fm}(R)=fR$. Thus, $\COH^0_{\fm}(R)\neq 0$ and that $\fm \COH^0_{\fm}(R)=0$. Consequently,
 $\COH^0_{\fm}(R)$ is a nonzero $k$-vector space.
Note that $\min(R)=\{(f_i):f_i \textit{ is an irreduciable component of  } f\}.$ Thus $R$ is equi-dimensional. We deduce from    Corollary \ref{m1} that
$\dim (\Syz_i(M))=\dim R$  for all $i>2$. If $\lambda(M)<\infty$, then we use Proposition  \ref{m} to observe that $\dim (\Syz_i(M))=\dim R$  for all $i>0$.
\end{proof}

\begin{fact}\label{Ramras}(See\cite[Theorem 3.2]{ram})
Let $I$ be an  ideal in a normal local ring $(S,\fn)$  which is not contained in any  height one prime. Set $R:=S/I\fn$. Let $M$ be  finitely generated and non-free. Then $M$
 has  strictly increasing betti numbers.
\end{fact}

\begin{fact}\label{f5}(See \cite{ext})
i) If $\beta_i(M)>\beta_{i-1}(M)$, then $\Supp(\Syz_{i+1}(M))
 =\Spec(R)$. In particular, $\dim(\Syz_{i+1}(M))=\dim R$. Suppose on the contradiction that there is a $\fp\in\Spec(R)\setminus\Supp(\Syz_{i+1}(M))$. We may assume that $\fp\in\min(R)$. Thus
$\ker(f_i)_{\fp}=\Syz_{i+1}(M)_{\fp}=0$. So $(f_i)_{\fp}:R_{\fp}^{\beta_{i}}\to R_{\fp}^{\beta_{i-1}}$ is injective. This contradicts  $\beta_i(M)>\beta_{i-1}(M)$. Similarly:

 ii)  If $\beta_i(M)<\beta_{i-1}(M)$, then $\Supp(\Syz_{i-1}(M))
 =\Spec(R)$. In particular, $\dim(\Syz_{i-1}(M))=\dim R$.
 \end{fact}

\begin{remark}
Let us recall that the results are in the realm of commutative rings. We just present a funny point: a ring $A$ is said to have invariant basis number property if $A^n\simeq A^m$ implies that $n=m$ for all $n$ and $m$. There are rings without  invariant basis number property.
\end{remark}

\begin{example}
Let $\fp$ be a height two prime ideal in   $k[[X_1,\ldots,X_n]]$  and let $R:=k[[X_1,\ldots,X_n]]/\fp\fm^t$ for some $t\geq0$. Let $M$ be finitely generated and non-free.
Then $\Supp(\Syz_i(M))=\Spec (R)$  for all $i>1$. If $\lambda(M)<\infty$, then the same claim holds for $i=1$.
\end{example}

\begin{proof}
If $t=0$, then $R=k[[X_1,\ldots,X_n]]/\fp$ is an integral domain. In view of Fact \ref{oki}  we get the claim. Thus, we may assume that $t>0$.  Set $I:=\fp \fm^{t-1}$. Then $I$ is not contained in any  height one prime ideal. In the light of Fact \ref{Ramras}, $\beta_{i+1}(M)>\beta_{i}(M)$. Due to  Fact \ref{f5}, $\Spec(\Syz_i(M))=\Spec (R)$  for all $i>1$. Without loss of generality we assume that $\dim R>0$. Now if $\lambda(M)<\infty$, in view of $0\to \Syz_1(M)\to R^{\beta_{1}}\to M\to 0$, we get that $\dim (\Syz_1(M))=\dim R$. By Lemma \ref{o}, we have $\Supp(\Syz_1(M))=\Spec (R)$.
\end{proof}

  When is $\lambda(\Syz_{2}(M))<\infty$?
If such a thing happens for a finite length module, then the ring is $1$-dimensional  with a nonzero nilpotent.



\begin{thebibliography}{99}

\bibitem{b}Kristen A.
Beck, and Micah J. Leamer,  \emph{Asymptotic behavior of dimensions of syzygies}, Proc. AMS {\bf{141}} (2013), 2245-–2252.

\bibitem{BH}
W. Bruns and J. Herzog,  \emph{Cohen-Macaulay rings}, Cambridge University Press, {\bf{39}}, Cambridge, (1998).

\bibitem{int}
S. Choi, \emph{
Betti numbers and the integral closure of ideals},
Math. Scand. {\bf{66}} (1990), no. 2, 173-–184.

\bibitem{ext}
A. Crabbe, D. Katz,  J. Striuli, and E. Theodorescu,  \emph{Hilbert-Samuel polynomials for the contravariant extension functor}, Nagoya Math. J. {\bf{198}} (2010), 1–-22.

\bibitem{h}
A. De Stefani, C. Huneke, and L. Nunez-Betancourt, \emph{Frobenius betti numbers and modules of finite
projective dimension}, Journal of commutative algebra {\bf{9}} (4), (2017) 455 --490.


\bibitem{Ok}S. Okiyama, \emph{A local ring is CM if and only if its residue field has a CM syzygy}, Tokyo J.
Math. {\bf{14}} (1991), 489--500.

\bibitem{ram2}
Eugene H. Gover, M. Ramras, \emph{Increasing sequences of Betti numbers}, Pacific J. Math.  {\bf{87}} (1980),  65-–68.

\bibitem{ram}
M. Ramras, \emph{Betti numbers and reflexive modules}, Ring theory (Proc. Conf., Park City, Utah, 1971),   297–-308. Academic Press, New York, 1972.

\bibitem{int}
P. Roberts, \emph{Le th\'{e}or\`{e}me d' intersection}, C. R. Acad. Sci. Paris Ser. I Math.,
{\bf{304}}  (1987),   177–-180.


\bibitem{bus}J.
St\"{u}ckrad, and W. Vogel,  \emph{Buchsbaum rings and applications. An interaction between algebra, geometry, and topology}, Mathematische Monographien {\bf{21}}, VEB Deutscher Verlag der Wissenschaften, Berlin, 1986.


\end{thebibliography}
\end{document}